\documentclass[11pt]{amsart}
\usepackage{mathrsfs}
\usepackage{amssymb}
\usepackage{amsmath}
\usepackage{amsthm}

\makeatletter
\def\LaTeX{\leavevmode L\raise.42ex
    \hbox{\kern-.3em\size{\sf@size}{0pt}\selectfont A}\kern-.15em\TeX}
\makeatother

\newcommand{\BibTeX}{{\rm B\kern-.05em{\sc
          i\kern-.025emb}\kern-.08em\TeX}}

\makeatletter
\def\@currentlabel{2.1}\label{e:dispaa}
\def\@currentlabel{2.21}\label{e:dispau}
\def\@currentlabel{2.22}\label{e:dispav}
\def\@currentlabel{2.23}\label{e:dispaw}
\def\@currentlabel{2.24}\label{e:dispax}
\def\theequation{\thesection.\@arabic\c@equation}
\makeatother

\numberwithin{equation}{section}

\newcounter{mnotecount}[section]

\newcommand{\rmnote}[1]{}

\usepackage{amsfonts,latexsym}

\usepackage{graphics}

\renewcommand{\theequation}{\arabic{section}.\arabic{equation}}

\newtheorem{thm}{Theorem}[section]
\newtheorem{lem}[thm]{Lemma}

\newtheorem{prop}[thm]{Proposition}
\theoremstyle{definition}

\newtheorem{rem}[thm]{Remark}

\begin{document}
\title[Classification of positive radial solutions to a weighted biharmonic equation]{Classification of positive radial solutions to a weighted biharmonic equation}

\author{Yuhao Yan}
\address{School of Mathematical Sciences, East China Normal
University, Shanghai 200241, China}
\email{yanyuhao2021@163.com}

\date{}

\begin{abstract}
\noindent In this paper, we consider the weighted fourth order equation
$$\Delta(|x|^{-\alpha}\Delta u)+\lambda \text{div}(|x|^{-\alpha-2}\nabla u)+\mu|x|^{-\alpha-4}u=|x|^\beta u^p\quad \text{in} \quad \mathbb{R}^n \backslash \{0\},$$
where $n\geq 5$, $-n<\alpha<n-4$, $p>1$ and $(p,\alpha,\beta,n)$ belongs to the critical hyperbola
$$\frac{n+\alpha}{2}+\frac{n+\beta}{p+1}=n-2.$$
We prove the existence of radial solutions to the equation for some $\lambda$ and $\mu$. On the other hand, let $v(t):=|x|^{\frac{n-4-\alpha}{2}}u(|x|)$, $t=-\ln |x|$, then for the radial solution $u$ with non-removable singularity at origin, $v(t)$ is a periodic function if $\alpha \in (-2,n-4)$ and $\lambda$, $\mu$ satisfy some conditions; while for $\alpha \in (-n,-2]$, there exists a radial solution with non-removable singularity and the corresponding function $v(t)$ is not periodic. We also get some results about the best constant and symmetry breaking, which is closely related to the Caffarelli-Kohn-Nirenberg type inequality. 
\end{abstract}

\maketitle

\section{Introduction}
In recent years, there has been a lot of interest in the weighted fourth order H{\'e}non equation
$$\Delta(|x|^{-\alpha}\Delta u)=|x|^{\beta}|u|^{p-1}u \quad \text{in} \quad \mathbb{R}^n,$$
where the parameters $\alpha$, $\beta$, $p$ and dimension $n$ are asked to satisfy certain conditions. For $\alpha=\beta=0$, it is closely related to the Q-curvature problem in comformal geometry. Some fundamental results about this equation have been established, such as the weighted Sobolev inequalities, existence of solutions, Caffarelli-Kohn-Nirenberg type inequalities and so on. We refer the readers to \cite{CC16,CM11,CM12,GHWW,GWW,HY,M14} and the references therein.

    In this paper, we are interested in the positive radial solutions of the equation
\begin{equation}\label{1.1}
\Delta(|x|^{-\alpha}\Delta u)+\lambda \text{div}(|x|^{-\alpha-2}\nabla u)+\mu|x|^{-\alpha-4}u=|x|^\beta u^p\quad \text{in} \quad \mathbb{R}^n \backslash \{0\},
\end{equation}
where $n\geq 5$, $-n<\alpha<n-4$, $p>1$ and $(p,\alpha,\beta,n)$ belongs to the critical hyperbola
\begin{equation}\label{1.2}
\frac{n+\alpha}{2}+\frac{n+\beta}{p+1}=n-2.
\end{equation}

  When $\alpha=\beta=\lambda=\mu=0$, equation (1.1) can be reduced to
\begin{equation}\label{1.3}
\Delta^2 u=u^{\frac{n+4}{n-4}} \quad \text{in} \quad  \mathbb{R}^n \backslash \{0\},
\end{equation}
which is conformal invariant and has a specific meaning in conformal geometry. It can also be derived from the Sobolev embedding of $H^2$ into $L^{\frac{2n}{n-4}}$
\begin{equation}\label{1.4}
\inf\limits_{u\in H^2(\mathbb{R}^n)} \frac{\int|\Delta u|^2}{\left(\int u^{\frac{2n}{n-4}}\right)^{\frac{n-4}{n}}}.
\end{equation}

  Lions proved the existence of extremal functions of (1.4) in \cite{LI84}. In the same paper, Lions also proved the radial symmetry of any extremal function of (1.4).
  In general, Lin \cite{LI98} proved that solutions of (1.3) with removable singularity at origin are given by
 $$u(x)=C_{n} \left(\frac{\lambda}{1+\lambda^2|x-x_0|^2}\right)^{\frac{n-4}{2}},\quad C_{n} =((n-4)(n-2)n(n+2))^{(n-4)/8},$$
where $\lambda>0$, $x_0\in\mathbb{R}^n.$ Moreover, let
$$v_{\lambda}(t):=|x-x_0|^{\frac{n-4}{2}}u_{\lambda}(|x-x_0|),\quad t=-\ln(\lambda |x-x_0|),$$
direct computation yields that
$$v_{\lambda}(t)=C_{n}(2\cosh t)^{-\frac{n-4}{2}}\quad \text{in} \quad\mathbb{R}.$$

 Lin also proved solutions of (1.3) with non-removable singularity at origin are radial symmetric. Passing to the Emden-Fowler coordinates and writting 
$$u(x)=|x|^{-\frac{n-4}{2}}v(t),\quad t=-\ln|x|,$$ 
Frank-K{\"o}nig \cite{FK19} and Guo-Huang-Wang-Wei \cite{GHWW} have showed that the function $v(t)$ is periodic in $t$.

    Using the above transformation, equation (1.3) becomes
\begin{equation}\label{1.5}
v^{(4)}-\frac{n(n-4)+8}{2} v''+\frac{n^2(n-4)^2}{16}v=v^{\frac{n+4}{n-4}}\quad \text{in} \quad \mathbb{R}.
\end{equation}
 Frank-K{\"o}nig proved that, if $v\in C^4(\mathbb{R})$ is a solution of (1.5), then one of the following three alternatives holds:\\
(a)$v\equiv l$(some positive constant) or $v\equiv 0;\\$
(b)$v(t)=C_{n}^{'}(2\cosh(t-T))^{-(n-4)/2},T\in R;\\$
(c)$v$ is periodic and has a unique local maximum and minimum per period and is symmetric with respect to its local extrema.

Thus, all radial solutions of equation (1.3) have been classified.

    Very recently, Huang-Wang \cite{HW20} considered the classification of the positive radial solutions of equation
\begin{equation}\label{1.6}
\Delta(|x|^{-\alpha}\Delta u)=|x|^{\beta}u^p\quad \text{in} \quad \mathbb{R}^n\backslash \{0\},
\end{equation}
where $-n<\alpha <n-4$, $p>1$ and $(p,\alpha,\beta,n)$ satisfies (1.2). Setting $v(t)=|x|^{\frac{n-4-\alpha}{2}}u(|x|)$, $t=-\ln |x|$, one obtains
\begin{equation}\label{1.7}
v^{(4)}-\frac{(n-2)^2+(\alpha+2)^2}{2} v''+\frac{(n-4-\alpha)^2(n+\alpha)^2}{16}v=v^p\quad \text{in} \quad\mathbb{R},
\end{equation}
as before, the solutions of (1.7) are given by one of the following three types:\\
($a^{'}$)$v\equiv 0$ or $v\equiv l;$\\
($b^{'}$)$v=C(\cosh \nu t)^m;$\\
($c^{'}$)$v$ is periodic.\\
  It is worth noting that, only when $\alpha=\beta=0$ or $(n+\alpha)(n+\beta)=(n-4-\alpha)^2$, (1.7) admits $C(\cosh\nu t)^m$-type solutions. We denote the corresponding solutions of (1.6) by $\overline{u}_1 (x)$ and $\overline{u}_2 (x)$.

  Moreover, Huang-Wang proved that the best constant
$$\overline{S}_{p}^{rad}(\alpha,\beta):=\inf\limits_{u\in \overline{\mathcal{N}}_{rad}^{2}(\mathbb{R}^n,\alpha)\backslash\{0\}}\frac{\int_{\mathbb{R}^n}|x|^{-\alpha}|\Delta u|^2 dx}{\left(\int_{\mathbb{R}^n}|x|^{\beta}|u|^{p+1}dx\right)^{2/(p+1)}}$$
is achieved in $\overline{\mathcal{N}}_{rad}^{2}(\mathbb{R}^n,\alpha)=\{u|u(x)=u(|x|),u(x)\in \overline{\mathcal{N}}^{2}(\mathbb{R}^n,\alpha)\}$, where $\overline{\mathcal{N}}^{2}(\mathbb{R}^n,\alpha)$ is the completion of $C_{0}^{\infty}(\mathbb{R}^n)$ in the norm
\begin{equation}
\||u\||_{\alpha}^{2}:=\int_{\mathbb{R}^n}|x|^{-\alpha}|\Delta u|^2dx.
\end{equation}
In particular, if $\alpha=\beta>-2$, $\overline{S}_{p}^{rad}(\alpha ,\alpha)$ is achieved by $\overline{u}_1$; if $(n+\alpha)(n+\beta)=(n-4-\alpha)^2$ with $\alpha<-2$, $\overline{S}_{p}^{rad}(\alpha ,\beta)$ is achieved by $\overline{u}_2$. Huang-Wang also proved, if $-2<\alpha<n-4$ and $u\in C^4(\mathbb{R}^n\backslash\{0\})$ is a positive solution of (1.6) with non-removable singularity at origin, then $v(t)(=|x|^{\frac{n-4-\alpha}{2}}u(|x|),t=-\ln(|x|))$ is periodic.
  
  In this paper, we generalize some results in \cite{HW20}, consider the positive radial solutions of equation (1.1), let $v(t)=|x|^{\frac{n-4-\alpha}{2}}u(|x|)$, $t=-\ln(|x|)$ , we have the ODE
\begin{equation}\label{1.8}
v^{(4)}-K_2 v''+K_0v=v^p\quad \text{in} \quad\mathbb{R},
\end{equation}
where $K_2 =\frac{(n-2)^2+(\alpha+2)^2}{2}-\lambda$, $K_0=\frac{(n-4-\alpha)^2(n+\alpha)^2}{16}-\lambda(\frac{n-4-\alpha}{2})^2+\mu.$\\
Denote
\begin{equation}\label{1.9}
\|u\|_\alpha ^2:=\int_{\mathbb{R}^n}|x|^{-\alpha}|\Delta u|^2-\lambda|x|^{-\alpha -2}|\nabla u|^2+\mu |x|^{-\alpha -4}u^2dx,
\end{equation}
then we will prove that when $\lambda,\mu$ satisfies
\begin{equation}\label{1.10}
\lambda \leq \frac{4\mu}{(n-4-\alpha)^2}+\frac{(n+\alpha)^2}{4}, \quad \mu \leq (\frac{n-4-\alpha}{2})^4,
\end{equation}
or
\begin{equation}\label{1.11}
\lambda \leq \frac{(n-4-\alpha)^2}{4}+\frac{(n+\alpha)^2}{4}, \mu > \left (\frac{n-4-\alpha}{2} \right )^4,
\end{equation}
$\|\cdot\|_\alpha$ is a norm of $C_{0,rad}^{\infty}(\mathbb{R}^n\backslash \{0\})=\{u|u\in C_{0}^{\infty}(\mathbb{R}^n\backslash \{0\}), u=u(|x|)\}$, when $\lambda,\mu$ satisfies (1.11) or (1.12), we define $\mathcal{N}_{rad}^{2}(\mathbb{R}^n,\alpha)$ as the completion of $C_{0,rad}^{\infty}(\mathbb{R}^n\backslash \{0\})$ in the norm $\|\cdot\|_\alpha$.

In Huang-Wang's paper \cite{HW20}, it is easy to see $\|| \cdot \||_{\alpha}$ is a norm. But in our case, $\lambda$ and $\mu$ are indeterminate, hence the positive definiteness and subadditivity don't necessarily satisfy. So we
need to use some Hardy-Rellich type equalities to transform the expression of $\|u\|_\alpha ^2$, and search for the range of $\lambda$, $\mu$ that guarantee the positive definiteness and subadditivity.

To elaborate our main results, we state three conditions first:

$(C_1)$ $\mu \geq 0$ and $\lambda \leq (n-2)(2+\alpha)-2 \sqrt{\mu}$;

$(C_2)$ $0 \leq \mu \leq (\frac{n-4-\alpha}{2})^{4}$ and $(n-2)(2+\alpha)+2\sqrt{\mu}\leq \lambda \leq \frac{4\mu}{(n-4-\alpha)^2}+\frac{(n+\alpha)^2}{4}$;

$(C_3)$ $\mu \leq 0$ and $\lambda \leq \frac{4\mu}{(n-4-\alpha)^2}+\frac{(n+\alpha)^2}{4}$.\\
Notice that if $\lambda$, $\mu$ satisfies any of these three conditions, then $K_0 \geq 0$, $K_2 \geq 0$, $K_2^2-4K_0\geq 0$ and (1.11)-(1.12) hold.

The main results of this paper are as follows:

\begin{thm}
Assume $-n<\alpha<n-4$, $p>1$ , $(p,\alpha,\beta,n)$ satisfies (1.2) and one of the conditions $(C_1)-(C_3)$ holds. Then in $\mathcal{N}_{rad}^{2}(\mathbb{R}^n,\alpha)$, the problem (1.1) admits a unique positive solution $u$. Moreover,\\
i) If $\alpha =\beta >-2$, and the following hold
\begin{equation}\label{1.12}
\frac{\lambda}{(n-2)^2}=\frac{(p+1)^4-16\pm \sqrt{[(p+1)^4-16]^2+(p+1)^2(p+3)^4[(p+1)^2+4]^2\frac{\mu}{(n-2)^4}}}{2(p+1)^2(p+3)^2},
\end{equation}
then the unique solution can be explicitly given by
$$u(x)=u_1(x):=\frac{C_1}{2^{m_1}}|x|^{-(\frac{n-4-\alpha}{2}+\nu_1 m_1)} (1+|x|^{2\nu_1})^{m_1}$$
with 

$m_1=-\frac{n-4-\alpha}{2+\alpha}$, $\nu_1=\frac{2+\alpha}{2} \sqrt{1-\frac{2\lambda}{(n-2)^2+(2+\alpha)^2}}$, $C_1^{p-1}=m_1(m_1-1)(m_1-2)(m_1-3)\nu_1^{4}$;\\
ii) If $(n+\alpha)(n+\beta)=(n-4-\alpha)^2$ with $\alpha <-2$ and the following hold
\begin{equation}\label{1.13}
\frac{\lambda}{(n-2)^2}=\frac{(p+1)^2[16-(p+1)^4]\pm \sqrt{(p+1)^4[(p+1)^4-16]^2+16(p+1)^2(p+3)^4[(p+1)^2+4]^2\frac{\mu}{(n-2)^4}}}{8(p+1)^2(p+3)^2},
\end{equation}
then the unique solution is given by
$$u(x)=u_2(x):=\frac{C_2}{2^{m_2}}|x|^{-(\frac{n-4-\alpha}{2}+\nu_2 m_2)}(1+|x|^{2\nu_2})^{m_2},$$
with 

$m_2=\frac{n+\alpha}{2+\alpha}$, $\nu_2=-\frac{2+\alpha}{2}\sqrt{1-\frac{2\lambda}{(n-2)^2+(2+\alpha)^2}}$, $C_2^{p-1}=m_2(m_2-1)(m_2-2)(m_2-3)\nu_2^{4}.$
\label{thm-1}
\end{thm}
 \vspace{0.3 cm}

As we know, the best constant is usually not easy to calculate. But in our work, for some special cases, the best constant can be calculated, the main results of this are as follows.
  Denote
$$S_{p}^{rad}(\alpha,\beta)=\inf \limits_{u \in \mathcal{N}_{rad}^{2}(\mathbb{R}^n,\alpha) \backslash \{0\} }\frac{\int_{\mathbb{R}^n}|x|^{-\alpha}|\Delta u|^{2}-\lambda |x|^{-\alpha-2}|\nabla u|^{2}+\mu |x|^{-\alpha-4}u^{2}dx}{(\int_{\mathbb{R}^{n}}|x|^{\beta}|u|^{p+1}dx)^{\frac{2}{p+1}}},$$
then we have

\begin{thm}
Under the conditions of theorem 1.1, the following hold\\
i) $S_p^{rad}(\alpha,\beta)>0$, and $S_p^{rad}(\alpha,\beta)$ is achieved in  $\mathcal{N}_{rad}^{2}(\mathbb{R}^n,\alpha)$; \\
ii) If $\alpha=\beta>-2$, and (1.13) is satisfied, then $S_p^{rad}(\alpha,\alpha)=\omega_n^{\frac{2(2+\alpha)}{n+\alpha}}\phi_{1,p}(\alpha)$ is achieved by $u_1(x)$,where 
$$\phi_{1,p}(\alpha)=\nu_1^{\frac{2(2n+\alpha-2)}{n+\alpha}}m_1(m_1-1)(m_1-2)(m_1-3)[\frac{4m_1(m_1-1)}{(2m_1-1)(2m_1-3)}B(-m_1,\frac{1}{2})]^{\frac{2(2+\alpha)}{n+\alpha}};$$
iii) If $(n+\alpha)(n+\beta)=(n-4-\alpha)^2$ with $\alpha <-2$, and (1.14) is satisfied, then $S_p^{rad}(\alpha,\beta)=\omega_n^{-\frac{2(2+\alpha)}{n-4-\alpha}}\phi_{2,p}(\alpha)$ is achieved by $u_2(x)$, where
$$\phi_{2,p}(\alpha)=\nu_2^{\frac{2(2n-\alpha-6)}{n-4-\alpha}}m_2(m_2-1)(m_2-2)(m_2-3)[\frac{4m_2(m_2-1)}{(2m_2-1)(2m_2-3)}B(-m_2,\frac{1}{2})]^{-\frac{2(2+\alpha)}{n-4-\alpha}}.$$ 

\label{thm-2}
\end{thm}

  For the radial solutions of (1.1) with non-removable singularity at origin, we have the following periodicity results.
\begin{thm}
Assume $-2<\alpha<n-4$, $\lambda \leq (n-2)(\alpha+2)-2\sqrt{\mu}$, $\mu>0$, $u\in C^{4}(\mathbb{R}^{N}\backslash\{0\})$ is a positive radial solution of equation (1.1) with non-removable singularity at origin. Define $v(t)=|x|^{\frac{n-4-\alpha}{2}}u(|x|)$ with $t=-\ln(|x|)$, then\\
i) $v$ is periodic;\\
ii) Let $a \in (0,l)$, $l:=K_0^{\frac{1}{p-1}}$, then there exists a unique (up to translations) periodic solution $v$ of (1.9) with minimal value $a$.

\label{thm-3}
\end{thm}

We have stated a result about the best constant in Theorem 1.2, as for the best constant, we usually concern about whether the function that achieves it is radially symmetric or not. So in the second part of this paper, we study the symmetry breaking phenomena of the equation
\begin{equation}\label{1.11}
\Delta(|x|^{-\alpha}\Delta u)+\lambda \text{div}(|x|^{-\alpha-2}\nabla u)+\mu|x|^{-\alpha-4}u=|x|^\beta |u|^{p-1}u\quad \text{in} \quad \mathbb{R}^n \backslash \{0\}.
\end{equation}
Denote
$$S_{\alpha,p}(\lambda,\mu):=\inf \limits_{u\in D^{2,2}(\mathbb{R}^{n},\alpha)\backslash\{0\}}\frac{\int_{\mathbb{R}^{n}}|x|^{-\alpha}|\Delta u|^{2}-\lambda|x|^{-\alpha-2}|\nabla u|^{2}+\mu|x|^{-\alpha-4}u^2 dx}{(\int_{\mathbb{R}^{n}}|x|^{\beta}|u|^{p+1} dx)^{\frac{2}{p+1}}}$$
$$S_{\alpha,p}^{rad}(\lambda,\mu):=\inf \limits_{u\in D_{rad}^{2,2}(\mathbb{R}^{n},\alpha)\backslash\{0\}}\frac{\int_{\mathbb{R}^{n}}|x|^{-\alpha}|\Delta u|^{2}-\lambda|x|^{-\alpha-2}|\nabla u|^{2}+\mu|x|^{-\alpha-4}u^2 dx}{(\int_{\mathbb{R}^{n}}|x|^{\beta}|u|^{p+1} dx)^{\frac{2}{p+1}}}$$
where $D^{2,2}(\mathbb{R}^{n},\alpha)$ is the completion of $C_0^{\infty}(\mathbb{R}^{n}\backslash \{0\})$ in the norm $\||u\||_{\alpha}^{2}=\int_{\mathbb{R}^{n}}|x|^{-\alpha}|\Delta u|^2 dx.$
  Now we can write down our main result about the symmetry breaking.
\begin{thm}
Assume that $n \geq 5$, $\lambda<0$ and $\mu>\frac{(n-4-\alpha)^2}{4} \lambda$, then there exists $\bar{\alpha}>0$ and $\bar{p}\in (2,2^{**})$ depend only on $\lambda,\mu$,such that if $p\in (\bar{p}-1,2^{**}-1)$ and $|\alpha|<\bar{\alpha},$ then $S_{\alpha,p}(\lambda,\mu)<S_{\alpha,p}^{rad}(\lambda,\mu).$
\label{thm-4}
\end{thm}

  The paper is organized as follows. In Section 2, we will study the classification of positive radial solutions and prove Theorem 1.1-Theorem 1.3. In Section 3, we will study the symmetry breaking phenomena and prove two important lemmas, and at the end of this section, we prove Theorem 1.4.

\section{Classification of Positive Radial Solutions and Proofs of Theorem 1.1-Theorem 1.3}

    In this section, we will prove the first three theorems. Recall that we have defined
\begin{equation}
\|u\|_{\alpha} ^2:=\int_{\mathbb{R}^n}|x|^{-\alpha}|\Delta u|^2-\lambda |x|^{-\alpha -2}|\nabla u|^2+\mu |x|^{-\alpha -4}u^2 dx.
\end{equation}
Now we show that $\|\cdot\|_\alpha$ is a norm of $C_{0,rad}^{\infty}(\mathbb{R}^n\backslash \{0\})$ if $\lambda$, $\mu$ satisfies (1.11) or (1.12). For this purpose, we need the following results. 

Denote $\|f\|_{\alpha,1}^{2}:=\int_{\mathbb{R}^{n}}|x|^{-\alpha}f^2 dx$ and assume $u(x)=u(|x|)$, then we have
\begin{equation}\label{2.2}
\|\Delta u\|_{\alpha,1}^2=\frac{(n+\alpha)^2}{4}\|\nabla u\|_{\alpha+2,1}^{2}+\frac{(n-4-\alpha)^2}{4}\|T_{\alpha+2}u\|_{\alpha+2,1}^2+\|T_{\alpha}\circ T_{\alpha+2}u\|_{\alpha,1}^{2},
\end{equation}
and
\begin{equation}\label{2.3}
\|\nabla u\|_{\alpha,1}^2=\frac{(n-2-\alpha)^2}{4}\|u\|_{\alpha+2,1}^2+\|\nabla(r^{\frac{n-2-\alpha}{2}}u)\|_{n-2,1}^2,
\end{equation}
where $T_{\alpha}u=u'+\frac{n-2-\alpha}{2r}u$. The above two equalities can be found in \cite{HY}.

Direct computation shows 
\begin{equation}\label{2.4}
|\nabla(r^{\frac{n-4-\alpha}{2}}u)|=r^{\frac{n-4-\alpha}{2}}|T_{\alpha+2}u|,\quad\quad \|\nabla(r^{\frac{n-4-\alpha}{2}}u)\|_{n-2,1}^2=\|T_{\alpha+2}u\|_{\alpha+2,1}^2,
\end{equation}
combining with (2.3), we obtain
\begin{equation}\label{2.5}
\|\nabla u\|_{\alpha+2,1}^2=\frac{(n-4-\alpha)^2}{4}\|u\|_{\alpha+4,1}^2+\|T_{\alpha+2}u\|_{\alpha+2,1}^2.
\end{equation}

We claim

\begin{align*}
\|u\|_{\alpha}^2 &=\left [\mu+\frac{(n+\alpha)^2(n-4-\alpha)^2}{16}-\frac{(n-4-\alpha)^2}{4} \lambda \right ]\|u\|_{\alpha+4,1}^2 \nonumber\\
& +\left [\frac{(n+\alpha)^2}{4}-\lambda +\frac{(n-4-\alpha)^2}{4} \right ]\|T_{\alpha+2} u\|_{\alpha+2,1}^2\nonumber\\
& +\|T_{\alpha}\circ T_{\alpha+2} u\|_{\alpha,1}^2.\nonumber\\
\end{align*}
In fact
\begin{align*}
\|u\|_{\alpha}^2 & =\|\Delta u\|_{\alpha,1}^2-\lambda\|\nabla u\|_{\alpha+2,1}^2+\mu\|u\|_{\alpha+4,1}^2\notag \\
& =\frac{(n+\alpha)^2}{4}\|\nabla u\|_{\alpha+2,1}^{2}+\frac{(n-4-\alpha)^2}{4}\|T_{\alpha+2}u\|_{\alpha+2,1}^{2}+\|T_{\alpha} \circ T_{\alpha+2}u\|_{\alpha,1}^{2}\notag \\ 
& -\lambda \|\nabla u\|_{\alpha+2,1}^{2}+\mu \|u\|_{\alpha+4,1}^{2}\notag \\
& =\left (\frac{(n+\alpha)^2}{4}-\lambda \right )\|\nabla u\|_{\alpha+2,1}^{2}+\frac{(n-4-\alpha)^2}{4}\|T_{\alpha+2}u\|_{\alpha+2,1}^2+\|T_{\alpha}\circ T_{\alpha+2} u\|_{\alpha,1}^2\notag \\
& +\mu \|u\|_{\alpha+4,1}^2\notag \\
& =\left (\frac{(n+\alpha)^2}{4}-\lambda \right )\left [\frac{(n-4-\alpha)^2}{4}\|u\|_{\alpha+4,1}^2+\|T_{\alpha+2} u\|_{\alpha+2,1}^2 \right ]\notag \\
& +\frac{(n-4-\alpha)^2}{4}\|T_{\alpha+2}u\|_{\alpha+2,1}^2+\|T_{\alpha}\circ T_{\alpha+2} u\|_{\alpha,1}^2+\mu \|u\|_{\alpha+4,1}^2.\notag \\
\end{align*}

It is easy to see our claim is true, so if $\lambda \leq \frac{4\mu}{(n-4-\alpha)^2}+\frac{(n+\alpha)^2}{4}$ and $\mu \leq (\frac{n-4-\alpha}{2})^4$ or $\lambda \leq \frac{(n-4-\alpha)^2}{4}+\frac{(n+\alpha)^2}{4}$ and $\mu > \left (\frac{n-4-\alpha}{2} \right )^4$, $ \|\cdot\|_{\alpha}$ is a norm of $C_{0,rad}^{\infty}(\mathbb{R}^{n}\backslash \{0\}).$

 \vspace{0.3 cm}

To prove Theorem 1.1, we need the following key result, which was proved in \cite{BM12}:
\begin{thm}
Let $q>1$ and let $A$, $B$ are two given positive constants. Then\\
i) The ordinary differential equation
\begin{equation}
w^{(4)}-Aw''+Bw=|w|^{q-1}w\quad \text{on}\quad\mathbb{R}
\end{equation}
has at least a nontrivial solution $w\in H^2(\mathbb{R})$. More precisely, $w$ achieves
$$\inf_{w\in H^2(\mathbb{R})\backslash\{0\}}\frac{\int|w''|^2+A|w'|^2+Bw^2 ds}{(\int|w|^{q+1} ds)^{2/(q+1)}}.$$
ii) If in addition $A^2-4B\geq 0$, then $w$ is the unique (up to translations, inversion $s \mapsto -s$ and change of sign) nontrivial solution to (2.8) in $H^2(\mathbb{R})$. Moreover, $w$ can be taken to be even, positive and strictly decreasing on $\mathbb{R}_{+}$.

\end{thm}

\noindent\textbf{Proof of Theorem 1.1}
Since $K_2^2-4K_0\geq 0$, we can apply Theorem 2.1 to get the existence and uniqueness of the positive radial solution of (1.1) in $\mathcal{N}_{rad}^{2}(\mathbb{R}^{n},\alpha)$. For the $C(\cosh \nu t)^{m}$-type solution, we set

$$v(t)=C(\cosh \nu t)^{m},$$
where $C,\nu$ and $m$ are to be determined later. Since $\cosh(\cdot)$ is even, so we can assume that $\nu >0$, direct computation shows
$$v''(t)=C\nu^{2}m^2 (\cosh\nu t)^{m}-Cm(m-1)\nu^{2}(\cosh\nu t)^{m-2}$$
and

\begin{align*}
v^{(4)}(t) &=C\nu^{4}m^4 (\cosh\nu t)^{m}-Cm(m-1)[m^2+(m-2)^2]\nu^{4}(\cosh\nu t)^{m-2} \\
& +Cm(m-1)(m-2)(m-3)\nu^{4}(\cosh\nu t)^{m-4}.\\
\end{align*}
Hence,

\begin{align*}
& v^{(4)}(t)-K_{2}v''(t)+K_{0}v(t)\\
& =C[\nu^{4}m^{4}-K_{2}m^{2}\nu^{2}+K_{0}](\cosh\nu t)^{m}-Cm(m-1)\nu^{2}[\left(m^2+(m-2)^{2} \right)\nu^{2}-K_{2}](\cosh \nu t)^{m-2} \\
& +Cm(m-1)(m-2)(m-3)\nu^{4}(\cosh \nu t)^{m-4}.
\end{align*}

We choose $\nu$ such that 
$$\nu^{2}=\frac{K_{2}}{m^2+(m-2)^{2}},$$
then if 
$$\nu^{4}m^{4}-K_{2}m^{2}\nu^{2}+K_{0}=0,$$
and
$$Cm(m-1)(m-2)(m-3)\nu^{4}(\cosh \nu t)^{m-4}=C^{p}(\cosh \nu t)^{mp},$$
$v$ is a solution. That is
$$m=-\frac{4}{p-1}, \quad C^{p-1}=m(m-1)(m-2)(m-3)\nu^{4},$$
and
$$\frac{K_{2}^{2}}{K_{0}}=\frac{[(p+1)^2+4]^2}{4(p+1)^2}.$$

When $\alpha=\beta>-2$ and $\lambda$, $\mu$ satisfies (1.13), or $(n+\alpha)(n+\beta)=(n-4-\alpha)^{2}$ with $\alpha<-2$ and $\lambda$, $\mu$ satisfies (1.14), the above equality holds.

For such a radial solution we have
$$|x|^{\frac{n-4-\alpha}{2}}u(|x|)=C(\cosh \nu t)^{m},\quad t=-\ln |x|,$$
i.e.
$$u(|x|)=\frac{C}{2^{m}}|x|^{-\gamma}(1+|x|^{2\nu})^{m},$$
where $\gamma=\frac{n-4-\alpha}{2}+\nu m$. The equality (1.2) implies that $p+1=\frac{2(n+\beta)}{n-4-\alpha}$.

\textit{Case 1.} If $\alpha=\beta>-2$ and (1.13) holds, then
$$m_{1}=-\frac{n-4-\alpha}{2+\alpha},\quad \nu_{1}=\frac{2+\alpha}{2}\sqrt{1-\frac{2\lambda}{(n-2)^2+(2+\alpha)^2}},$$
$$C_{1}^{p-1}=m_{1}(m_{1}-1)(m_{1}-2)(m_{1}-3)\nu_{1}^{4}.$$
Therefore
$$u_{1}(x)=\frac{C_{1}}{2^{m_{1}}}|x|^{-(\frac{n-4-\alpha}{2}+\nu_1 m_1)}(1+|x|^{2\nu_1})^{m_{1}}.$$

\textit{Case 2.} If $(n+\alpha)(n+\beta)=(n-4-\alpha)^2$ with $\alpha<-2$ and (1.14) holds, then
$$m_{2}=\frac{n+\alpha}{2+\alpha},\quad \nu_{2}=-\frac{2+\alpha}{2}\sqrt{1-\frac{2\lambda}{(n-2)^2+(2+\alpha)^2}},$$

$$C_{2}^{p-1}=m_{2}(m_{2}-1)(m_{2}-2)(m_{2}-3)\nu_{2}^{4}.$$
Thus

$$u_{2}(x)=\frac{C_{2}}{2^{m_{2}}}|x|^{-(\frac{n-4-\alpha}{2}+\nu_2 m_2)}(1+|x|^{2\nu_2})^{m_{2}}.$$
\hfill $\square$\par
\begin{rem}
In fact, if we use (1.2) and replace $p$ with $n$ and $\alpha$, (1.13) and (1.14) can be uniformly written as the following equality:
$$\lambda=\frac{[(n-2)^2+(2+\alpha)^2]\left[(n-2)(2+\alpha)(n-4-\alpha)+\sqrt{(n-2)^2(2+\alpha)^2(n-4-\alpha)^2+4(n+\alpha)^2\mu}\right]}{(n+\alpha)^2(n-4-\alpha)}.$$
\end{rem} \vspace{0.5 cm}
As mentioned above, in general case, the best constant $S_{rad}^{p}(\alpha,\beta)$ is not easy to calculate. But in the special case, we can calculate it, and in the proof of Theorem 1.2, we use the method in Huang-Wang's paper \cite{HW20} to calculate the best constant.  \vspace{0.3 cm}

\noindent\textbf{Proof of Theorem 1.2}
To any radial function $u\in \mathcal{N}_{rad}^2(\mathbb{R}^n,\alpha)$, we associate a function $v\in H^2(\mathbb{R})$ via the Emden-Fowler transform. Thus by direct caculation, there holds
$$u(x)=|x|^{-\frac{n-4-\alpha}{2}}v(-\ln|x|),$$
$$\int_{\mathbb{R}^n}|x|^{\beta}|u|^{p+1}dx=\omega_n\int_{-\infty}^{\infty}|v|^{p+1}dt,$$
and
$$\int_{\mathbb{R}^n}|x|^{-\alpha}|\Delta u|^2dx=\omega_n\int_{-\infty}^{\infty}|v''|^2+K_2|v'|^2+K_0v^2dt,$$
where $\omega_n$ is the measure of $\mathbb{S}^{n-1}$. In particular,
$$S_{p}^{rad}(\alpha,\beta)=\omega_n^{\frac{p-1}{p+1}}\phi_p(\alpha),$$
where
$$\phi_p(\alpha)=\inf_{v\in H^2(\mathbb{R})\backslash\{0\}}\frac{\int_{-\infty}^{\infty}|v''|^2+K_2|v'|^2+K_0 v^2 dt}{(\int_{-\infty}^{\infty}|v|^{p+1} dt)^{\frac{2}{p+1}}}.$$

By Theorem 2.1, $\phi_p(\alpha)$ hence $S_{p}^{rad}(\alpha,\beta)$ is achieved.

By Theorem 1.1, we have that for $\alpha=\beta>-2$, or $(n+\alpha)(n+\beta)=(n-4-\alpha)^2$ with $\alpha<-2$, then $v=C(\cosh \nu t)^m$. Direct caculation shows that for any number $\gamma$,

\begin{align*}
\int_{0}^{\infty}(\cosh \nu t)^{\gamma}dt &=\nu^{-1}\int_{0}^{\infty}(\cosh t)^{\gamma} dt=\nu^{-1}\int_{1}^{\infty}\frac{x^{\gamma}}{\sqrt{x^2-1}}dx\\
& =\nu^{-1}\int_{0}^{1}y^{-\gamma-1}\frac{1}{\sqrt{1-y^2}}dy=\frac{\nu^{-1}}{2}\int_{0}^{1}x^{-\frac{2+\gamma}{2}}(1-x)^{-\frac{1}{2}}dx\\
& =\frac{\nu^{-1}}{2}B(-\frac{\gamma}{2},\frac{1}{2})=\frac{\nu^{-1}}{2}\frac{\Gamma(-\frac{\gamma}{2})\Gamma(\frac{1}{2})}{\Gamma({\frac{1-\gamma}{2}})},\\
\end{align*}

so there holds
\begin{align*}
\int_{-\infty}^{\infty}|v|^{p+1}dt &=C^{p+1}\int_{0}^{\infty}(\cosh \nu t)^{m(p+1)}dt=C^{p+1}\nu^{-1}\frac{\Gamma(-\frac{m(p+1)}{2})\Gamma(\frac{1}{2})}{\Gamma(\frac{1-m(p+1)}{2})}\\
& =C^{p+1}\nu^{-1}\frac{4m(m-1)}{(2m-1)(2m-3)}\frac{\Gamma(-m)\Gamma(\frac{1}{2})}{\Gamma(\frac{1}{2}-m)}.
\end{align*}

On the other hand,
\begin{align*}
  & |v''|^2+K_2|v'|^2+K_0 v^2\\
=& 2K_2 C^2 \nu^2 m^2 (\cosh \nu t)^{2m}-2m^2 C^2 \nu^4[2m^2-3m+2](\cosh \nu t)^{2m-2}\\
+& C^2 m^2(m-1)^2\nu^4(\cosh \nu t)^{2m-4}.
\end{align*} 
By complex and tedious caculation, we get
$$\int_{-\infty}^{\infty}|v''|^2+K_2|v'|^2+K_0 v^2 dt=C^2 \nu^2 m^2 \frac{4(m-1)^2(m-2)(m-3)}{(2m-1)(2m-3)}\frac{\Gamma(-m)\Gamma(\frac{1}{2})}{\Gamma(\frac{1}{2}-m)}.$$
Therefore,
$$\phi_p(\alpha)=\nu^{3+\frac{2}{p+1}}m(m-1)(m-2)(m-3)\left[\frac{4m(m-1)}{(2m-1)(2m-3)}B(-m,\frac{1}{2})\right]^{\frac{p-1}{p+1}}.$$
\textit{Case 1.} For $\alpha=\beta>-2$ and (1.13) holds, using
$$\nu_1=\frac{2+\alpha}{2}\sqrt{1-\frac{2\lambda}{(n-2)^2+(2+\alpha)^2}},\quad m_1=-\frac{n-4-\alpha}{2+\alpha},\quad p+1=\frac{2(n+\alpha)}{n-4-\alpha}.$$
Then
$$\phi_{1,p}(\alpha)=\nu_1^{\frac{2(2n+\alpha-2)}{n+\alpha}}m_1(m_1-1)(m_1-2)(m_1-3)\left[\frac{4m_1(m_1-1)}{(2m_1-1)(2m_1-3)}B(-m_1,\frac{1}{2})\right]^{\frac{2(2+\alpha)}{n+\alpha}},$$
$$S_p^{rad}(\alpha,\alpha)=\omega_n^{\frac{2(2+\alpha)}{n+\alpha}}\phi_{1,p}(\alpha),$$
and $S_p^{rad}(\alpha,\alpha)$ is achieved by 
$$u_{1}(x)=\frac{C_{1}}{2^{m_{1}}}|x|^{-(\frac{n-4-\alpha}{2}+\nu_1 m_1)}(1+|x|^{2\nu_1})^{m_{1}}.$$
\textit{Case 2.} For $(n+\alpha)(n+\beta)=(n-4-\alpha)^2$ with $\alpha<-2$ and (1.14) holds, using
$$\nu_2=-\frac{2+\alpha}{2}\sqrt{1-\frac{2\lambda}{(n-2)^2+(2+\alpha)^2}},\quad m_2=\frac{n+\alpha}{2+\alpha},\quad p+1=\frac{2(n-4-\alpha)}{n+\alpha},$$
we have
$$\phi_{2,p}(\alpha)=\nu_2^{\frac{2(2n-\alpha-6)}{n-4-\alpha}}m_2(m_2-1)(m_2-2)(m_2-3)\left[\frac{4m_2(m_2-1)}{(2m_2-1)(2m_2-3)}B(-m_2,\frac{1}{2})\right]^{-\frac{2(2+\alpha)}{n-4-\alpha}},$$
and 
$$S_p^{rad}(\alpha,\beta)=\omega_n^{-\frac{2(2+\alpha)}{n-4-\alpha}}\phi_{2,p}(\alpha),$$
which is achieved by
$$u_{2}(x)=\frac{C_{2}}{2^{m_{2}}}|x|^{-(\frac{n-4-\alpha}{2}+\nu_2 m_2)}(1+|x|^{2\nu_2})^{m_{2}}.$$
\hfill $\square$\par

\vspace{0.5 cm}

To prove Theorem 1.3, we need the following proposition about the behavior of $v$ on $\mathbb{R}$. Suppose $u \in C^{4}(\mathbb{R}^{N}\backslash \{0\})$ is any positive radial solution of (1.1) with non-removable singularity at origin. Define $v(t)=|x|^{\frac{n-4-\alpha}{2}}u(|x|)$ with $t=-\ln|x|$, then $v(t)$ satisfies the equation (1.9).

\begin{prop}
Assume that $-2<\alpha<n-4$, $\lambda<(n-2)(\alpha+2)-2\sqrt{\mu}$, $\mu\geq 0$ and $v\in C^{4}$ is a positive entire solution to (1.9), then equation $v'(t)=0$ admits infinitely many roots in $\mathbb{R}$.
\end{prop}

\begin{proof}

Define the energy function

\begin{equation}
E_{v}(t):=-v'(t)v'''(t)+\frac{1}{2}[v''(t)]^2+\frac{K_2}{2}[ v'(t)]^{2}-\frac{K_{0}}{2}v^2+\frac{1}{p+1}v^{p+1},
\end{equation}

then

\begin{equation}
\frac{dE_{v}(t)}{dt}=-v'(t)[v^{(4)}(t)-K_2 v''(t)+K_0 v-v^{p}]\equiv 0,
\end{equation}

i.e. $E_v (t)\equiv \mu_0.$
Assume that $v'(t)=0$ admits only finite roots in $\mathbb{R}$, then for $t$ large, $v$ is monotonous.
Denote
\begin{equation*}
v^{(4)}(t)-K_2 v''(t)=v^{p}(t)-K_0 v(t),
\end{equation*}
let $g(v)=v^{p}-K_0 v.$
Thus, by lemma 2.1 of \cite{HW20}, we have $v \to 0$ or $v \to l:=K_{0}^{\frac{1}{p-1}}$ as $t \to \pm \infty.$

Set $h(\pm \infty):=\lim_{t \to \pm \infty}h(t).$

\textit{Case 1.} If $v(-\infty)=0$, $v(+\infty)=l$. By lemma 2.2 of \cite{HW20}, we have
$$E_v (-\infty)=0, \quad E_v(+\infty)= -\frac{K_0}{2}l^2+\frac{1}{p+1} l^{p+1}=l^2 K_0(\frac{1}{p+1}-\frac{1}{2})<0,$$
which contradicts to (2.8).\\

\textit{Case 2.} If $v(-\infty)=l$, $v(+\infty)=0$, this case is similar to \textit{Case 1.} In fact, we can consider $v(-t).$

\textit{Case 3.} If $v(\pm \infty)=l$. As $v\not\equiv l$, there exists a global maximal point or minimal point $t_0$.
By (2.8), we have $E_v(+\infty)=E_v(t_0).$\\

On the other hand, $E_v (+\infty)=\frac{1}{p+1}l^{p+1}-\frac{K_0}{2} l^2=:G(l)$, with\\
$$G(s):=\frac{1}{p+1}s^{p+1}-\frac{K_0}{2}s^2.$$
$G'(s)=s^{p}-K_0 s=s(s^{p-1}-K_0)=s(s^{p-1}-l^{p-1})$ implies $G'(s)<0$ if $s<l$ and $G'(s)>0$ if $s>l$. Hence,
$$E_v(t_0)= \frac{1}{2}[v''(t_0)]^2+G(v(t_0))>\frac{1}{2}[v''(t_0)]^2+G(l) \geq G(l).$$ 
Contradiction.

\textit{Case 4.} If $v(\pm \infty)=0$. By direct computation, we have that the equation (1.9) has four eigenvalues
$$\lambda_1=\sqrt{\frac{K_2 +\sqrt{(\lambda-(n-2)(\alpha+2))^2-4\mu}}{2}},\quad \lambda_2=\sqrt{\frac{K_2 -\sqrt{(\lambda-(n-2)(\alpha+2))^2-4\mu}}{2}},$$
$$\lambda_3=-\sqrt{\frac{K_2 +\sqrt{(\lambda-(n-2)(\alpha+2))^2-4\mu}}{2}},\quad \lambda_4=-\sqrt{\frac{K_2 -\sqrt{(\lambda-(n-2)(\alpha+2))^2-4\mu}}{2}}.$$
Obviously, we have $\lambda_1>\lambda_2>0$, $\lambda_3<\lambda_4<0$.
By the variation of parametres method, the solution $v(t)$ to (1.9) is given by
\begin{align*}
v(t)& =\sum_{i=1}^4A_{i}e^{\lambda_{i}t}+\sum_{i=1}^4B_{i}\int_{0}^{t}e^{\lambda_{i}(t-s)}v^{p}(s) ds\\
& =A_1^{'}e^{\lambda_{1}t}+A_2^{'}e^{\lambda_2 t}+\sum_{i=3}^4 A_{i} e^{\lambda_{i} t}-\sum_{i=1}^2 B_i\int_{t}^{\infty}e^{\lambda_i(t-s)}v^{p}(s)ds\\
& +\sum_{i=3}^4 B_i \int_{0}^{t}e^{\lambda_i (t-s)}v^{p}(s)ds,\\
\end{align*}
where we have used the fact $e^{-\lambda_1 s}v^p(s)$, $e^{-\lambda_2 s}v^p(s)\in L^1(\mathbb{R}_+)$. Since $\lim_{t \to \infty}v(t)=0$, $A'_1=A'_2=0(\lambda_1 > \lambda_2 >0)$. Hence, there exists a constant $C>0$ such that for any $t \geq 0$, there holds
\begin{equation}
|v(t)|\leq Ce^{\lambda_4 t}+C\int_t^{\infty}e^{\lambda_2 (t-s)}|v^p(s)|ds+C\int_0^t e^{\lambda_4 (t-s)}|v^{p}(s)|ds,
\end{equation}
and for any $\delta >0$, there exists $M>0$, such that $|v^{p}(s)|\leq \delta |v(s)|$ if $s\geq M$. Then for any $t\geq M$,
\begin{align}
|v(t)| & \leq O(e^{\lambda_4 t})+C\delta \int_t^{\infty} e^{\lambda_2(t-s)}|v(s)|ds+C\delta \int_M^{t}e^{\lambda_4(t-s)}|v(s)|ds\nonumber\\
& =O(e^{\lambda_4 t})+C\delta K_1(t)+C\delta K_2(t),
\end{align}
where $K_1(t):=\int_M^t e^{\lambda_4(t-s)}|v(s)|ds$, $K_2(t):=\int_t^{\infty} e^{\lambda_2(t-s)}|v(s)|ds.$
Using (2.10), if we fix $\delta>0$ small enough such that $2C\delta \leq \min(\lambda_2,-\lambda_4)$, then we have
\begin{align*}
(K_1-K_2)'(t) & =|v(t)|+\lambda_4 K_1(t)-(-|v(t)|+\lambda_2 K_2(t))\\
& =2|v(t)|+\lambda_4 K_1(t)-\lambda_2 K_2(t)\\
& \leq 2C\delta(K_1+K_2)+O(e^{\lambda_4 t})+\lambda_4 K_1(t)-\lambda_2 K_2(t)\\
& \leq O(e^{\lambda_4 t}).
\end{align*}
As $\lim_{t \to \infty}v(t)=0$, we have readily
$$\lim_{t \to \infty}K_1(t)=\lim_{t \to \infty}K_2(t)=0.$$
Hence, $(K_2-K_1)(t)\leq O(e^{\lambda_4}t)(t \to \infty)$,
combine with (2.12), then
$$|v(t)| \leq O(e^{\lambda_4 t})+2C\delta K_1(t),$$
therefore,
$$K'_1(t)=|v(t)| +\lambda_4 K_1(t) \leq O(e^{\lambda_4 t})+(2C\delta+\lambda_4)K_1(t).$$
Thus,
$$K_1(t)=O(e^{(\lambda_4+2C\delta)t}).$$
By ODE theory, we have $|v(t)|=O(e^{(\lambda_4 +2C\delta)t}).$
Backing to (2.9), we have
$v(t)=O(e^{\lambda_4 t})$. Since $-\lambda_4>\frac{n-4-\alpha}{2}$, we get $u=O(1)$, which contradicts to that origin is non-removable singular point.

\end{proof}

Now we are ready to prove Theorem 1.3. \vspace{0.3 cm}

\noindent\textbf{Proof of Theorem 1.3}
$i)$ By Proposition 2.3, there are two real numbers $t_0<t_1$ such that $v'(t_0)=v'(t_1)=0$, $v'(t)\neq 0$ in $(t_0,t_1)$. By Lemma 2.5 of \cite{HW20}, $v$ must be periodic. That is, for any radial solution $u$ with
a non-removable singularity at origin, the function $v(-\ln|x|)=|x|^{\frac{n-4-\alpha}{2}}u(x)$ is periodic in $\ln |x|$.\\
$ii)$ The proof of this part is same as the proof of Theorem 1 in \cite{FK19}.    
\hfill $\square$\par

 \vspace{0.3 cm}

For the solution of a weighted equation, it is often difficult to determine whether its singularity is removable or not, the following proposition tells us that the singularity of the solution to (1.1) is non-removable if $\alpha$, $\lambda$, $\mu$ satisfy suitable condition.

\begin{prop}
Let $v(t)$ be the solution of (1.9)(which is given by Theorem 2.1), $u(|x|)=|x|^{-\frac{n-4-\alpha}{2}}v(t)$, $t=-\ln|x|$. If $-n<\alpha\leq -2$ and $\lambda >(n-2)(2+\alpha)+2\sqrt{\mu}$, then $u(0)=+\infty.$
\end{prop}

\begin{proof}
As the proof of proposition 2.3, we have
$$v(t)=\sum_{i=1}^{4}A_{i}e^{\lambda_{i}t}+\sum_{i=1}^{2}B_{i}\int_{t}^{\infty}e^{\lambda_{i}(t-s)}v^{p}(s)ds+\sum_{i=3}^{4}B_{i}\int_{-\infty}^{t}e^{\lambda_{i}(t-s)}v^{p}(s)ds.$$
Since $\lim_{t\to \pm \infty}v(t)=0$, we have $A_{1}=A_{2}=A_{3}=A_{4}=0$. The fact $v(t)=v(-t)$ implies that $B_{1}=B_{3}$, $B_{2}=B_{4}.$

We argue by contradiction. Assume $u(0)=O(1)$, which means $v(t)=O(e^{-\frac{n-4-\alpha}{2}t})$. Since $\lambda >(n-2)(2+\alpha)+2\sqrt{\mu}$, we have $-\lambda_4 < \frac{n-4-\alpha}{2}$, thus $B_{2}=B_{4}=0$. Therefore,
$$v(t)=B\int_{t}^{\infty}e^{\lambda_{1}(t-s)}v^{p}(s)ds+B\int_{-\infty}^{t}e^{\lambda_{3}(t-s)}v^{p}(s)ds,\quad for\quad some\quad B>0.$$
By direct computation, there holds 
$$v'(t)=B\lambda_{1}(\int_{t}^{\infty}e^{\lambda_{1}(t-s)}v^{p}(s)ds-\int_{-\infty}^{t}e^{\lambda_{3}(t-s)}v^{p}(s)ds),$$
$$v''(t)=-2B\lambda_{1}v^{p}(t)+\lambda_{1}^2 v(t),$$
$$v'''(t)=-2B\lambda_{1}pv^{p-1}(t)v'(t)+\lambda_{1}^2 v'(t),$$
$$v^{(4)}(t)=-2B\lambda_{1}p(p-1)v^{p-2}(t)[v'(t)]^2 -2B\lambda_{1}pv^{p-1}(t)v''(t)+\lambda_{1}^2 v''(t).$$
Since $K_{2}=\lambda_{1}^2+\lambda_{2}^2$, $K_{0}=\lambda_{1}\lambda_{2}$ and $v^{(4)}-K_{2}v''+K_{0}v=v^p$, we get
$$-2B\lambda_{1}p(p-1)(\frac{v'(t)}{v(t)})^2+4B^2\lambda_{1}^2 pv^{p-1}(t)-2B\lambda_{1}^3 p+2B\lambda_{1}\lambda_{2}^2=1.$$
For the case of $t=0$, $v'(0)=0$ and $t\to \infty$, $v(t)\to 0$, we have
$$4B^2 \lambda_{1}^2 pv^{p-1}(0)-2B\lambda_{1}^3 p+2B\lambda_{1}\lambda_{2}^2=1,$$
$$-2B\lambda_{1}^3 p+2B\lambda_{1}\lambda_{2}^2\geq 1,$$
which leads to
$$4B^2\lambda_{1}^2 pv^{p-1}(0)\leq 0.$$
A contradiction!

\end{proof}

\section{Symmetry Breaking and Proof of Theorem 1.4}

In this section, we will prove Theorem 1.4 by using the method used in [4], we first prove two lemmas and then, the conclusion of Theorem 1.4 is obvious.

Denote
$$S^{**}:=\inf \limits_{u\in C_{0}^{\infty}(\mathbb{R}^n)\backslash\{0\}}\frac{\int_{\mathbb{R}^{n}}|\Delta u|^2 dx}{(\int_{\mathbb{R}^{n}}|u|^{2^{**}}dx)^{\frac{2}{2^{**}}}},$$
where $2^{**}=\frac{2n}{n-4}$.

\begin{lem}
Assume $\lambda \leq 0$ and $\mu \geq \frac{(n-4)^2}{4} \lambda$, then $S_{0,2^{**}-1}(\lambda,\mu)=S^{**}$. Moreover, for any $\lambda <0$ and $\mu>\frac{(n-4)^2}{4} \lambda$, $S_{0,2^{**}-1}(\lambda,\mu)$ can not be achieved in $D^{2,2}(\mathbb{R}^{n}).$

\label{lem-1}
\end{lem}
\begin{proof}
Using the Hardy inequality:
\begin{equation}
\int_{\mathbb{R}^n}|x|^{-\alpha-2}|\nabla u|^2 dx \geq \frac{(n-4-\alpha)^2}{4}\int_{\mathbb{R}^n}|x|^{-\alpha-4} u^2 dx,
\end{equation}
we can easily get $S_{0,2^{**}-1}(\lambda,\mu)\geq S^{**}$.\\
  
  On the other hand, for any $u\in C_{0}^{\infty}(\mathbb{R}^{n})\backslash \{0\}$, let $u_{y}(x)=u(x-y)$, then
\begin{align*}
S_{0,2^{**}-1}(\lambda,\mu) & \leq \lim_{|y|\to \infty}\frac{\int_{\mathbb{R}^{n}}|\Delta u_{y}|^2 dx-\lambda \int_{\mathbb{R}^{n}}|x|^{-2}|\nabla u_{y}|^2 dx+\mu\int_{\mathbb{R}^{n}}|x|^{-4}|u_{y}|^2 dx}{(\int_{\mathbb{R}^{n}}|u_{y}|^{2^{**}})^{\frac{2}{2^{**}}}}\\
& =\lim_{|y|\to \infty}\frac{\int_{\mathbb{R}^{n}}|\Delta u|^2 dx-\lambda\int_{\mathbb{R}^{n}}|x+y|^{-2}|\nabla u|^2+\mu\int_{\mathbb{R}^{n}}|x+y|^{-4}|u|^2 dx}{(\int_{\mathbb{R}^{n}}|u|^{2^{**}}dx)^{\frac{2}{2^{**}}}}\\
& =\frac{\int_{\mathbb{R}^{n}}|\Delta u|^2 dx}{(\int_{\mathbb{R}^{n}}|u|^{2^{**}}dx)^{\frac{2}{2^{**}}}},\\
\end{align*}
thus,
$$S_{0,2^{**}-1}(\lambda,\mu)=S^{**}.$$
If $u\in D^{2,2}(\mathbb{R}^{n})$ is a minimizer of $S_{0,2^{**}-1}(\lambda,\mu)$, then $u$ is also a minimizer of $S^{**}$.
Therefore,
$$\left(\mu-\frac{(n-4)^2}{4} \lambda \right)\int_{\mathbb{R}^{n}}|x|^{-4}|u|^2 dx=0,$$
which implies $u=0$, a contradiction.
\end{proof}

\begin{lem}
For any $\lambda \leq 0$ and $\mu \geq \frac{(n-4-\alpha)^2}{4} \lambda$, the following inequalities hold
$$\limsup_{(\alpha,p)\to(0,(2^{**}-1)_{-})}S_{\alpha,p}(\lambda,\mu)\leq S_{0,2^{**}-1}(\lambda,\mu),$$
$$S_{0,2^{**}-1}^{rad}(\lambda,\mu)\leq \liminf_{(\alpha,p)\to(0,(2^{**}-1)_{-})}S_{\alpha,p}^{rad}(\lambda.\mu).$$
\end{lem}

\begin{proof}
Fix $\lambda \leq 0$ and $\mu \geq \frac{(n-4-\alpha)^2}{4} \lambda$, for any $ u\in C_{0}^{\infty}(\mathbb{R}^{n} \backslash \{0\})$, denote
$$Q_{\alpha,p}(u)=\frac{\int_{\mathbb{R}^{n}}|x|^{-\alpha}|\Delta u|^2 dx-\lambda \int_{\mathbb{R}^{n}}|x|^{-\alpha-2}|\nabla u|^2 dx+\mu \int_{\mathbb{R}^{n}}|x|^{-\alpha-4}u^2 dx}{(\int_{\mathbb{R}^{n}}|x|^{\beta}|u|^{p+1}dx)^{\frac{2}{p+1}}}.$$
As when $(\alpha,p)\to(0,(2^{**}-1)_{-})$, we have $Q_{\alpha,p}(u)\to Q_{0,2^{**}-1}(u)$, and $C_{0}^{\infty}(\mathbb{R}^{n}\backslash \{0\})$ is dense in $D^{2,2}(\mathbb{R}^{n};\alpha)$, the first inequality immediately follows.

  On the other hand, for any $p \in (1,2^{**}-1)$, let $\theta_p =\frac{p-1}{2^{**}-2}$. By the H{\"o}lder inequality, one has that
$$\int_{\mathbb{R}^{n}}|x|^{\beta}|u|^{p+1}dx\leq(\int_{\mathbb{R}^{n}}|x|^{-\alpha-4}|u|^2dx)^{1-\theta_{p}}(\int_{\mathbb{R}^{n}}|x|^{-\frac{n\alpha}{n-4}}|u|^{2^{**}}dx)^{\theta_p}.$$
Therefore
\begin{align}
[S_{\alpha,p}^{rad}(\lambda,\mu)]^{\frac{p+1}{\theta_p 2^{**}}} 
& =\left[\inf \limits_{u\in D_{rad}^{2,2}(\mathbb{R}^{n},\alpha)\backslash\{0\}}\frac{\int_{\mathbb{R}^{n}}|x|^{-\alpha}|\Delta u|^{2}-\lambda|x|^{-\alpha-2}|\nabla u|^{2}+\mu|x|^{-\alpha-4}u^2 dx}{(\int_{\mathbb{R}^{n}}|x|^{\beta}|u|^{p+1} dx)^{\frac{2}{p+1}}}\right]^{\frac{p+1}{\theta_p 2^{**}}}\nonumber \\
& \geq \inf \limits_{u\in D_{rad}^{2,2}(\mathbb{R}^{n},\alpha)\backslash\{0\}} \frac{(\int_{\mathbb{R}^{n}}|x|^{-\alpha}|\Delta u|^2-\lambda |x|^{-\alpha-2}|\nabla u|^2+\mu |x|^{-\alpha-4}u^2 dx)^{\frac{p+1}{\theta_p 2^{**}}}}{(\int_{\mathbb{R}^{n}}|x|^{-\alpha-4}|u|^2)^{\frac{2(1-\theta_p)}{p+1} \frac{p+1}{\theta_p 2^{**}}}(\int_{\mathbb{R}^{n}}|x|^{-\frac{n \alpha}{n-4}}|u|^{2^{**}})^{\frac{2\theta_p}{p+1} \frac{p+1}{\theta_p 2^{**}}}}\nonumber \\
& \geq \eta_{\alpha}^{\frac{2(1-\theta_p)}{\theta_p 2^{**}}}S_{\alpha,2^{**}-1}^{rad}(\lambda,\mu),\nonumber\\
\end{align}
where $\eta_{\alpha}$ is the best constant of the weighted Rellich inequality
$$\eta_{\alpha} \int_{\mathbb{R}^{n}}|x|^{-\alpha-4}|u|^2 dx \leq \int_{\mathbb{R}^{n}}|x|^{-\alpha}|\Delta u|^2 dx, \quad \forall u\in C_{0}^{\infty}(\mathbb{R}^{n}\backslash\{0\}).$$
For any $\alpha$, let $\tau_{\alpha}:=1-\frac{\alpha}{n-4}$,for any radial $u \in C_{0}^{\infty}(\mathbb{R}^{n}\backslash\{0\}) $, set $\tilde{u}(x)=u(|x|^{1/\tau_{\alpha}})$, then we have
$$\int_{\mathbb{R}^{n}}|\tilde{u}|^{2^{**}}dx=\tau_{\alpha}\int_{\mathbb{R}^{n}}|x|^{-\frac{n\alpha}{n-4}}|u|^{2^{**}}dx,$$
$$\int_{\mathbb{R}^{n}}|x|^{-2}|\nabla \tilde{u}|^2 dx=\tau_{\alpha}^{-1}\int_{\mathbb{R}^{n}}|x|^{-\alpha-2}|\nabla u|^2 dx,$$
$$\int_{\mathbb{R}^{n}}|x|^{-4}\tilde{u}^2 dx=\tau_{\alpha}\int_{\mathbb{R}^{n}}|x|^{-\alpha-4}u^2 dx,$$
$$\int_{\mathbb{R}^{n}}|\Delta \tilde{u}|^2 dx=\tau_{\alpha}^{-3}\int_{\mathbb{R}^{n}}|x|^{-\alpha}|\Delta u+R_{\alpha}u|^2 dx,$$
where
$$R_{\alpha}u:=(\tau_{\alpha}-1)(n-2)\nabla u \cdot \frac{x}{|x|^2}.$$
Denote
$$\gamma_{\alpha}:=\inf \limits_{u\in D^{2,2}(\mathbb{R}^{n};\alpha)\backslash\{0\}}\frac{\int_{\mathbb{R}^{n}}|x|^{-\alpha}|\Delta u|^2 dx}{\int_{\mathbb{R}^{n}}|x|^{-\alpha-2}|\nabla u|^2 dx},$$
and set
$$\epsilon_{\alpha}=|\tau_{\alpha}-1|(n-2)\gamma_{\alpha}^{-1/2},$$
since $\alpha \in (-n,n-4)$, we have
$$\int_{\mathbb{R}^{n}}|x|^{-\alpha}|R_{\alpha}u|^2 dx \leq \epsilon_{\alpha}^2 \int_{\mathbb{R}^{n}}|x|^{-\alpha}|\Delta u|^2 dx.$$
By the Cauchy-Schwarz inequality, one obtains
$$\int_{\mathbb{R}^{n}}|x|^{-\alpha}|\Delta u+R_{\alpha}u|^2 dx \leq (1+\epsilon_{\alpha})^2\int_{\mathbb{R}^{n}}|x|^{-\alpha}|\Delta u|^2 dx.$$
Thus,
\begin{align*}
S_{0,2^{**}-1}^{rad}(\lambda,\mu) &\leq \frac{\int_{\mathbb{R}^{n}}(|\Delta \tilde{u}|^2-\lambda|x|^{-2}|\nabla \tilde{u}|^2+\mu|x|^{-4}\tilde{u}^2)dx}{(\int_{\mathbb{R}^{n}}|\tilde{u}|^{2^{**}}dx)^{2/2^{**}}}\\
& \leq \frac{(1+\epsilon_{\alpha})^2}{\tau_{\alpha}^{3+2/2^{**}}} \frac{\int_{\mathbb{R}^{n}}(|x|^{-\alpha}|\Delta u|^2-\lambda|x|^{-\alpha-2}|\nabla u|^2+\mu|x|^{-\alpha-4}u^2)dx}{(\int_{\mathbb{R}^{n}}|x|^{-\frac{n\alpha}{n-4}}|u|^{2^{**}}dx)^{2/2^{**}}}\\
& -\frac{\lambda}{\tau_{\alpha}^{1+2/2^{**}}}(1-\frac{(1+\epsilon_{\alpha})^2}{\tau_{\alpha}^2})\frac{\int_{\mathbb{R}^{n}}|x|^{-\alpha-2}|\nabla u|^2dx}{(\int_{\mathbb{R}^{n}}|x|^{-\frac{n\alpha}{n-4}}|u|^{2^{**}}dx)^{2/2^{**}}}\\
& +\mu \tau_{\alpha}^{1-\frac{2}{2^{**}}}(1-\frac{(1+\epsilon_{\alpha})^2}{\tau_{\alpha}^4})\frac{\int_{\mathbb{R}^{n}}|x|^{-\alpha-4}u^2 dx}{(\int_{\mathbb{R}^{n}}|x|^{-\frac{n\alpha}{n-4}}|u|^{2^{**}}dx)^{2/2^{**}}}\\
& \leq K_{\alpha} \frac{\int_{\mathbb{R}^{n}}(|x|^{-\alpha}|\Delta u|^2-\lambda |x|^{-\alpha-2}|\nabla u|^2+\mu |x|^{-\alpha-4}u^2)dx}{(\int_{\mathbb{R}^{n}}|x|^{-\frac{n\alpha}{n-4}}|u|^{2^{**}}dx)^{2/2^{**}}},\\
\end{align*}
where
$$K_{\alpha}=\frac{(1+\epsilon_{\alpha})^2}{\tau_{\alpha}^{3+2/2^{**}}}-\frac{\lambda \gamma_{\alpha}^{-1}}{\tau_{\alpha}^{1+2/2^{**}}}\left|1-\frac{(1+\epsilon_{\alpha})^2}{\tau_{\alpha}^2}\right|+|\mu| \eta_{\alpha}^{-1}\tau_{\alpha}^{1-2/2^{**}}\left|1-\frac{(1+\epsilon_{\alpha})^2}{\tau_{\alpha}^4}\right|.$$
Therefore,
\begin{equation}
S_{\alpha,2^{**}-1}^{rad}(\lambda,\mu) \geq K_{\alpha}^{-1}S_{0,2^{**}-1}^{rad}(\lambda,\mu).
\end{equation}
For $|\alpha|$ small, there exists a constant $C>0$, such that $C^{-1} \leq \gamma_{\alpha} \leq C$ and $C^{-1} \leq \eta_{\alpha} \leq C$. Consequently, as $(\alpha,p) \to (0,(2^{**}-1)_{-})$, we have $\tau_{\alpha} \to 1$, $\theta_{p} \to 1$, $\epsilon_{\alpha} \to 0$, $K_{\alpha} \to 1$. Hence, the second inequality follows from (3.2)-(3.3).
\end{proof}

\noindent\textbf{Proof of Theorem 1.4}
Fix $\lambda <0$ and $\mu >\frac{(n-4-\alpha)^2}{4} \lambda$, by Theorem 2.1, $S_{0,2^{**}-1}^{rad}(\lambda,\mu)$ is achieved in $D_{rad}^{2,2}(\mathbb{R}^{n})$. On the other hand, by Lemma 3.1, $S_{0,2^{**}-1}(\lambda,\mu)$ can not be achieved. Hence,
$S_{0,2^{**}-1}^{rad}(\lambda,\mu)>S_{0,2^{**}-1}(\lambda,\mu)$, then the conclusion follows by using Lemma 3.2. 
\hfill $\square$\par

\end{document}